\documentclass[twoside,10pt]{article} 
\usepackage{graphicx}
\usepackage{amsfonts}
\usepackage[fleqn]{amsmath}
\usepackage{amssymb}
\usepackage{amsthm}
\usepackage{amscd}
\usepackage[T1]{fontenc}
\usepackage{afterpage}  %
\usepackage{fancyhdr}
\usepackage{color}

\theoremstyle{plain}
\newtheorem{theorem}{Theorem}[section]
\newtheorem{lemma}[theorem]{Lemma}
\newtheorem{proposition}[theorem]{Proposition}

\newtheorem{corollary}[theorem]{Corollary}
\theoremstyle{definition}

\newtheorem{example}[theorem]{Example}
\theoremstyle{remark}

\begin{document}

\afterpage{\rhead[]{\thepage} \chead[\small W. A. Dudek and R. A. R. Monzo \ \ \ \ ]
{\small Double magma and Ward quasigroups \ \ \ \ \ \ \ } \lhead[\thepage]{} }                  

\begin{center}
\vspace*{2pt}
{\Large \textbf{Double magma associated with Ward}}\\[3mm]
{\Large\textbf{and double Ward quasigroups}}\\[26pt]
 {\large \textsf{\emph{Wieslaw A. Dudek \ and \ Robert A. R. Monzo}}}\\[26pt]
\end{center}
 {\footnotesize\textbf{Abstract.} We determine types of double magma associated with Ward quasigroups, double Ward quasigroups, their duals and the groups they generate. Ward quasigroup double magma and unipotent, right modular, left-unital double magma are proved to be improper. Necessary and sufficient conditions are found on a pair of right modular, left unital magma (and right-left unital magma) for them to form a double magma. We give further insight into the intimate connection between the property of mediality and the interchange law by proving that a quasigroup is medial if and only if any pair of its parastrophic binary operations satisfy the interchange law.} 
\footnote{ 

\hspace*{2.5mm}interchange law, parastrophe.}


\section*{\centerline{1. Introduction}}
Inspired by Ward's paper \cite{Ward} on postulating the inverse operations in groups Cardoso and de Silva defined in \cite{Car} the notion of a {\em Ward quasigroup} as a quasigroup $(Q,\cdot)$ containing an element $e$ such that $xx=e$ for all $x\in Q$, and satisfying the identity $xy\cdot z=x(z\cdot ey)$. Polonijo \cite{Pol2} proved that these two conditions can be replaced by the identity:
\begin{eqnarray}\label{e1}
&&xz\cdot yz=xy.
\end{eqnarray}  
 
Note that Ward quasigroups are uniquely determined by groups. Given a group $(G,\circ,^{-1},e)$, we can construct a Ward quasigroup by defining $xy=x \circ y^{-1}$. Conversely, given a Ward quasigroup $Q$, it can be shown that $Q$ is {\em unipotent} ($xx=yy$), so we may write $xx=e$, and defining $x^{-1}=ex$ and $x \circ y=xy^{-1}=x\cdot ey$ makes $(Q,\circ,^{-1},e)$ a group. If a group $(Q,\circ)$ is abelian, then the corresponding Ward quasigroup is medial and is a group-like BCI-algebra (cf. \cite{D'88}). Obviously such a quasigroup is a BCI-algebra satisfying the Is\'eki's condition $(S)$ and is the so-called p-semisimple part of each BCI-algebra (cf. \cite {D'86} or \cite{Ise}).

Writing der$(Q,\circ,e)$ for the Ward quasigroup derived (constructed) from the group $(Q,\circ,e)$ and ret$(Q,\cdot,e)$ for the group obtained from the Ward quasigroup $(Q,\cdot,e)$, it can also be shown that der(ret$(Q,\cdot,e))=(Q,\cdot,e)$ and ret(der$(Q,\circ,e))=(Q,\circ,e)$ (cf. \cite{Ch} and \cite{Pol}). Therefore, there is a one-to-one correspondence between groups and Ward quasigroups.

Similarly, given a group $(G,\circ,e)$, we can construct a quasigroup by defining $xy=x^{-1} \circ y^{-1}$.
This quasigroup satisfies the Ward-like identity 
\begin{eqnarray}\label{e2}
&&(ee\cdot xz)(ey\cdot z)=xy.
\end{eqnarray}
A quasigroup $(Q,\cdot)$ containing an element $e$ such that the above condition is satisfied for all $x,y,z\in Q$ is called a {\em double Ward quasigroup} (not to be confused with the Ward double quasigroups of \cite{Pol}) and is denoted by $(Q,\cdot, e)$.
Fiala proved (cf. \cite{Fia}) that similarly as in the case of Ward quasigroups there is one-to-one correspondence between groups and double Ward quasigroups.

Edmunds in \cite{Edm} constructed double magma in groups, using commutator operations. He found that a medial Ward quasigroup and its dual, an left unipotent, right modular quasigroup, form a double magma pair that is proper if and only if the group derived from the medial Ward quasigroup is not boolean (see Theorem \ref{T33} below). This example inspired the authors to search for double magma associated with Ward and double Ward quasigroups, their duals and the groups they generate, culminating in the results below. The question of whether there are double Ward double magma, and proper ones, remains unanswered.

\section*{\centerline{2. Preliminary definitions, notation and results}}\setcounter{section}{2}\setcounter{theorem}{0}

Let $(Q,\cdot)$ be a {\em magma}, i.e., a non-empty set $Q$ with one binary operation denoted by a dot or by juxtaposition. If in $(Q,\cdot)$ there is an element $r$ such that $xr=x$ for all $x\in Q$, then it is called a {\em right unit}. If there is $l\in Q$ such that $lx=x$ for all $x\in Q$, then it is called a {\em left unit}. An element which at the same time is a right and left unit is called a {\em unit}. A magma with a (right, left) unit is called ({\em right, left}) {\em unital}. If additionally, for all $x\in Q$ we have $xx=r$ ($xx=l$), then we say that $(Q,\cdot)$ is {\em $r$-unipotent} (resp. {\em $l$-unipotent}). A magma that is $r$-unipotent and $l$-unipotent (in this case $r=l=1$) is called {\em unipotent}.
  
A magma $(Q,\cdot)$ is called {\em right modular}, if $xy\cdot z=zy\cdot x$;
{\em left modular}, if $x\cdot yz=z\cdot yx$; {\em medial} or {\em entropic}, if $xy\cdot zw=xz\cdot yw$, and {\em reversible}, if $xy\cdot zw=wz\cdot yx$ hold for all $x,y,z,w\in Q$.
Both right and left modular magmas are medial. Magmas $(Q,\cdot)$ and $(Q,\bar{\cdot})$, where $x\,\bar{\cdot}\,y=y\cdot x$ are called {\em dual}.

Let $(Q,\cdot ,\ast)$ be a triple consisting of a non-empty set $Q$ and two binary operations satisfying one of the {\em interchange laws}:
\begin{eqnarray}\label{e3}
&&(x\cdot y)\ast(z\cdot w)=(x\ast z)\cdot (y\ast w),\\
\label{e4}
&&(x\cdot y)\ast(z\cdot w)=(y\ast x)\cdot (w\ast z), \\
\label{e5}
&&(x\cdot y)\ast(z\cdot w)=(w\ast z)\cdot (y\ast x).
\end{eqnarray}
$(Q,\cdot,\ast)$ satisfying the first law is called a {\em double magma}, the second -- a {\em lateral double magma}, the third -- a {\em reversible double magma}. Then $(Q,\cdot)$ and $(Q,\ast)$ are called {\em double magma partners} (or resp. {\em lateral double magma partners} and {\em reversible double magma partners}). If these operations are distinct, then we say that a double magma is {\em proper}. It is clear that $(Q,\cdot,\ast)$ is a double magma if and only if $(Q,\ast,\cdot)$ is a double magma.

In relation to two-fold monoidal categories Kock introduced in \cite{Koc} the notion of a {\em double semigroup}, i.e., a double magma with two associative operations. He proved that in cancellable double semigroups and inverse double semigroups both operations must be commutative. Moreover, if a commutative double semigroup is unital, then these two operations coincide.

\begin{example}\label{Ex1} Let $xy=[ax+by]_n$ and $x\ast y=[cx+dy]_n$ where $a,b,c,d\in \mathbb{Z}_n=\{1,2,\ldots,n\}$ and $[a+b]_n$ is calculated modulo $n$. Then, $(\mathbb{Z}_n,\cdot,\ast)$ is a double magma. In the case $a=b$, $c=d$ it is a commutative double semigroup.
\end{example}

\begin{example}\label{Ex2} Let $\mathbb{R}^+$ be the set of positive reals and $n\geq 2$ be an integer. Then $(\mathbb{R}^+,\cdot,\ast)$, where $x\cdot y=x^ny^n$ and $x\ast y=\sqrt[n]{xy}$, is a double magma. Also $(\mathbb{R},+,-)$ is a double magma.
\end{example}

Note that the interchange laws are symmetric, in the sense that any one of these laws is satisfied if and only if the same law, with $\cdot$ and $\ast$ interchanged, is satisfied. That is, we have 
$$
\begin{array}{cccc}
(x\cdot y)\ast(z\cdot w)=(x\ast z)\cdot (y\ast w)\Leftrightarrow (x\ast y)\cdot (z\ast w)=(x\cdot z)\ast (y\cdot w),\\[2pt]
(x\cdot y)\ast(z\cdot w)=(y\ast x)\cdot (w\ast z)\Leftrightarrow (x\ast y)\cdot (z\ast w)=(y\cdot x)\ast (w\cdot z),\\[2pt]
(x\cdot y)\ast(z\cdot w)=(w\ast z)\cdot (y\ast x)\Leftrightarrow (x\ast y)\cdot (z\ast w)=(w\cdot z)\ast (y\cdot x).
\end{array}
$$

 Partly as a consequence of this symmetry, all the results of this paper have their dual version.

A double magma $(Q,\cdot,\ast)$ is called ({\em right}, {\em left}) {\em unital} if both magma are (right, left) unital. It is called {\em $lr$-unital} if $(Q,\cdot)$ is left unital and $(Q,\ast)$ is right unital.

Note that in a Ward quasigroup $(Q,\cdot)$ there is an element $e$ (the unity of the corresponding group) such that

\arraycolsep=.5mm
\begin{eqnarray}
&&xx=yy=e,\label{e6}\\
&&xe=x,\label{e7}\\
&&e\cdot xy=yx,\label{e8}\\
&&e\cdot ex=x,\label{e9}\\
&&xy\cdot z=x(z\cdot ey)\label{e10}
\end{eqnarray}
hold for all $x,y,z\in Q$.

\medskip
With these definitions the proofs of the following two Lemmas are straightforward and are omitted.

\begin{lemma}\label{L23} The following statements are valid.
\begin{enumerate}
\item[$(1)$] $(Q,\cdot,\bar{\cdot}\,)$ is a double magma if and only if $(Q,\cdot)$ is medial,
\item[$(2)$] $(Q,\cdot,\bar{\cdot}\,)$ always satisfies the reversible interchange law,
\item[$(3)$] $(Q,\cdot,\bar{\cdot}\,)$ satisfies the lateral interchange law if and only if $(Q^2,\cdot)$ is commutative, 
\item[$(4)$] $(Q,\cdot,\ast)$ is a double magma if and only if $(Q,\,\bar{\cdot}\,,\bar{\ast})$ is a double magma, 
\item[$(5)$] $(Q,\cdot,\ast)$ is a lateral double magma if and only if $(Q,\bar{\cdot}\,,\bar{\ast})$ is a lateral double magma,   
\item[$(6)$] $(Q,\cdot,\ast)$ is a reversible double magma if and only if $(Q,\bar{\cdot}\,,\bar{\ast})$  is a reversible double magma,
\item[$(7)$] $(Q,\cdot,\cdot)$  satisfies the interchange law if and only if $(Q,\cdot)$ is medial.
	\end{enumerate}
\end{lemma}
\begin{lemma}\label{L24} The following statements are valid. 
\begin{enumerate}
\item[$(1)$] A Ward quasigroup $(Q,\cdot,e)$ is medial if and only if it is left modular if and only if {\rm ret}$(Q,\cdot,e)$ is abelian,  
\item[$(2)$] A Ward quasigroup $(Q,\cdot,e)$ is commutative if and only if {\rm ret}$(Q,\ast,e)$ is a boolean group,                                                                                                
\item[$(3)$]  $(Q,\cdot,e)$ is a Ward quasigroup if and only if there is a group $(Q,\star,e)$ such that $xy=y^{-1}\star x$.  
\end{enumerate}
\end{lemma}
\begin{proposition}\label{P25} If for a Ward quasigroup $(Q,\cdot,e)$ and a magma $(Q,\ast,l)$ with unique left unit $l$, $(Q,\cdot,\ast)$ is a double magma, then $l=e$.
\end{proposition}
\begin{proof} Indeed, $x=xe=(l\ast x)\cdot(l\ast e)=(l\cdot l)\ast (x\cdot e)=(l\cdot l)\ast x$. So, $l=l\cdot l=e$.
\end{proof}
 
The following is a well-known result. We give the proof to help make this paper more self-contained and also to give a flavour of some of the methods used in the proofs below. This result was proved in \cite{Eck} and uses what is known as the Eckmann-Hilton argument. 

\begin{theorem}\label{T26} If $(Q,\cdot,\ast)$ is a unital double magma with units $e$ and $e^{\!\ast}$ respectively, then these operations coincide and $(Q,\cdot,\cdot)$ is a commutative, unital and associative double magma.
\end{theorem}
\begin{proof} We have 
$$
x\cdot y=(x\ast e^{\!\ast})\cdot(e^{\!\ast}\ast y)=(x\cdot e^{\!\ast})\ast(e^{\!\ast}\cdot y)=(e^{\!\ast}\ast x)\cdot(y\ast e^{\!\ast})=(e^{\!\ast}\cdot y)\ast(x\cdot e^{\!\ast}).
$$
Thus $y\cdot y=(e^{\!\ast}\cdot y)\ast(y\cdot e^{\!\ast})$. But 
$$
e^{\!\ast}=e^{\!\ast}\ast e^{\!\ast}=(e^{\!\ast}\cdot e)\ast (e\cdot e^{\!\ast})=(e^{\!\ast}\ast e)\cdot (e\ast e^{\!\ast})=e\cdot e=e.
$$ So, $e^{\!\ast}=e$, and consequently $x\cdot y=y\ast x=x\ast y$. Hence $(Q,\cdot,\ast)=(Q,\cdot,\cdot)$ is a commutative double magma. In fact, it is a medial double magma, because, with $\cdot=\ast$, the interchange law implies that $\cdot=\ast$ is medial. Hence, $(x\ast y)\ast z=(x\ast y)\cdot (e^{\!\ast}\ast z)=(x\ast y)\cdot (e\ast z)=(x\cdot e)\ast (y\cdot z)=x\ast (y\ast z)$ and $(Q,\ast,\ast)=(Q,\cdot,\cdot)$ is an associative double magma. 
\end{proof}
Another way to describe this result is that a unital double magma $(Q,\cdot,\ast)$ is a commutative monoidal double semigroup $(Q,\cdot,\ast)=(Q,\cdot,\cdot)$.

\section*{\centerline{3. Double magma and Ward quasigroups}\setcounter{section}{3}\setcounter{theorem}{0}
}
As noted in \cite{Edm}, in light of Theorem \ref{T26}, in order to produce a proper double magma we must be sure that the double magma is not unital. It is tempting to call a medial pair of dual magma, which by Lemma \ref{L23} form a double magma, improper. However, we will not do so, as such double magma can give interesting examples (see Theorem \ref{T33} below, from \cite{Edm}). Given Theorem \ref{T26}, it is natural to ask whether there are proper right unital, left unital or $lr$-unital double magma. We now explore this question, first finding out whether there are Ward quasigroup double magma.    
                                                                              
\begin{theorem}\label{T31} Let $(Q,\cdot,e)$ and $(Q,\ast,e^{\!\ast})$ be Ward quasigroups. Then the following statements are equivalent:

\begin{enumerate}
\item[$(1)$] $(Q,\cdot,\ast)$  is a double magma,
\item[$(2)$]  $(Q,\cdot)=(Q,\ast)$ is medial,
\item[$(3)$] $(Q,\cdot)=(Q,\ast)$ is left modular,
\item[$(4)$] $(Q,\cdot)=(Q,\ast)$  satisfies $xy\cdot z=xz\cdot y$ and $x\cdot yz=xy\cdot ez$,
\item[$(5)$] $(Q,\cdot)=(Q,\ast)$ is induced by an abelian group.
\end{enumerate}
\end{theorem}
\begin{proof} $(1)\Rightarrow (2)$. Since $e=e^{\!\ast} e^{\!\ast}$ and $e^{\!\ast}=e\ast e$, we have $e^{\!\ast}=e\ast e=ee\ast ee=(e\ast e)\cdot (e\ast e)=e$, so $e=e^{\!\ast}$.
Thus, $x\ast y=(x\ast y)\cdot e^{\!\ast}=(x\ast y)\cdot (y\ast y)=(x\cdot y)\ast (y\cdot y)=(x\cdot y)\ast e^{\!\ast}=x\cdot y$. Therefore $(Q,\cdot)=(Q,\ast)$ and the mediality is obvious. 
 
$(2)\Rightarrow (1)$. It is obvious.

$(2)\Leftrightarrow (3)$. This follows from Lemmas \ref{L23} and \ref{L24}.

$(2)\Rightarrow (4)$. $xy\cdot z=xy\cdot ze=xz\cdot ye=xz\cdot y$ and $x\cdot yz=xe\cdot yz=xy\cdot ez$.

$(4)\Rightarrow (2)$. $xy\cdot zw=(xy\cdot z)\cdot ew=(xz\cdot y)\cdot ew=xz\cdot yw$.

$(2)\Leftrightarrow (5)$. This follows from the fact that $(Q,\cdot,e)={\rm der(ret}(Q,\cdot,e))$. 
\end{proof}

So, a medial Ward quasigroup double magma is improper and the only Ward quasigroup double magmas are medial and improper. However, a medial Ward quasigroup and its dual form a double magma, by Lemma \ref{L23}.

The next theorems are a consequence of Lemma \ref{L23}.

\begin{theorem}\label{T32} If $(Q,\cdot,e)$ is a medial Ward quasigroup, then $(Q,\cdot,\,\bar{\cdot}\,)$ is a double $rl$-unital magma and $(Q,\,\bar{\cdot}\,,e)$ is a unipotent, left unital and right modular quasigroup with unique left unit $e$.
\end{theorem}

\begin{theorem}\label{T33} {\rm (\cite{Edm}, p. 2--3)} If a Ward quasigroup $(Q,\cdot)$ is derived from an abelian group $(Q,\circ)$, then $(Q,\cdot,\,\bar{\cdot}\,)$ is a double $rl$-unital magma. It is proper if and only if $(Q,\circ)$ is not a boolean group.
\end{theorem}
\begin{theorem}\label{T34} If $(Q,\cdot,e)$ is a medial Ward quasigroup, then $(Q,\cdot,\circ)$, where $(Q,\circ)={\rm ret}(Q,\cdot,e)$, is a double magma. It satisfies the lateral inverse law if and only if it satisfies the reverse inverse law if and only if $(Q,\circ)$ is a boolean group.
\end{theorem}

\begin{theorem}\label{T35} If a Ward quasigroup $(Q,\cdot,e)$ forms a double magma with a unital magma $(Q,\circ,1)$, then $(Q,\circ,1)$ is a retract of $(Q,\cdot,e)$ and $(Q,\cdot,e)$ is medial. If $(Q,\cdot,e)$ is medial and {\rm ret}$(Q,\cdot,e)$ forms a double magma with a left cancellative, unipotent magma $(Q,\star)$ containing a right unit $r$, then $(Q,\cdot)=(Q,\star)$.
\end{theorem}
\begin{proof} If $(Q,\cdot,e)$ forms a double magma with a unital magma $(Q,\circ,1)$, then $e=ee=(e\circ 1)\cdot (1\circ e)=e1\circ 1e=e1\circ 1=e1=1$. So, $e=1$ and $x\circ y\stackrel{\eqref{e7},\eqref{e9}}{=}(e\cdot ex)\circ ye\stackrel{\eqref{e3}}{=}(e\circ y)\cdot (ex\circ e)=y\cdot ex=xe\circ(e\cdot ey)\stackrel{\eqref{e3}}{=}(x\circ e)(e\circ ey)=x\cdot ey$. So, $(Q,\circ)={\rm ret}(Q,\cdot,e)$ and $(Q,\circ,e)$ is an abelian group. Consequently, $(Q,\cdot)$ is medial.

Now, if $(Q,\cdot,e)$ is medial and $(Q,\circ,e)={\rm ret}(Q,\cdot,e)$ forms a double magma with a left cancellative, unipotent magma $(Q,\star)$ containing a right unit $r$, then $(Q,\circ,e)$ is abelian and $r=r\star r=(r\circ e)\star (e\circ r)=(r\star e)\circ (e\star r)=(r\star e)\circ e=r\star e$, whence, by a left cancellation, we obtain $r=e$. Thus, $x\star y=(e\circ x)\star (y\circ e)=(e\star y)\circ (x\star r)=(e\star y)\circ x=(e\star y)\cdot ex$. This for $y=x$ gives $(e\star x)\cdot ex=x\star x=r=e=ex\cdot ex$, which implies $e\star x=ex$ because $(Q,\cdot)$ is a quasigroup. So,  $x\star y=(e\star y)\cdot ex=ey\cdot ex=e\cdot yx\stackrel{\eqref{e8}}{=}xy$. Hence, $(Q,\cdot)=(Q,\star)$.
\end{proof}

\section*{\centerline{4. Double Ward quasigroups}}\setcounter{section}{4}\setcounter{theorem}{0}

Fiala proved (cf. \cite{Fia}) that double Ward quasigroups are in a one-to-one correspondence with groups. Namely, he proved that on any double Ward quasigroup $(Q,\cdot,e)$ we can defined a group Ret$(Q,\cdot,e)=(Q,\diamond)$ by putting $x\diamond y=ex\cdot ey$. In this group $e$ is the identity and $x^{-1}=ex$. On the other side, any group $(Q,\circ)$ can be transformed into a double Ward quasigroup $(Q,\ast,e)$ with the operation $x\ast y=x^{-1}\circ y^{-1}$. Such obtained quasigroup is denoted by Der$(Q,\circ,e)$. Moreover, similarly as in Ward quasigroups, we have Der(Ret$(Q,\cdot,e))=(Q,\cdot, e)$ and Ret(Der$(Q,\circ,e))=(Q,\circ,e)$.

Note that element $e$ used in the definition of a double Ward quasigroup is not uniquely determined. Indeed, since $xy=x^{-1}\circ y^{-1}$ for some group $(Q,\circ,1)$, from \eqref{e2} it follows that in this group $e^3=1$. So, if $(Q,\circ,1)$ is an abelian group, then as $e$ we can take the unity of this group or an arbitrary element of order $3$.

Let $(Q,\cdot,e)=(Q,\cdot,r)$, $e\ne r$, be double Ward quasigroups and $x\diamond y=ex\cdot ey=\alpha (x)\cdot\alpha(y)$, $x\bullet y=rx\cdot ry=\beta(x)\cdot\beta(y)$, where $\alpha,\beta$ are bijections of $Q$. Then $(Q,\diamond,e)$ and $(Q,\bullet,r)$ are two groups. Since $\alpha^{-1}\beta(x)\diamond\alpha^{-1}\beta(y)=\beta(x)\cdot\beta(y)=x\bullet y$, these groups are isotopic. So, by the Albert's theorem, they are isomorphic. Obviously, isomorphic groups determine isomorphic double Ward quasigroups. Thus, double Ward quasigroups are isotopic if and only if they are isomorphic.

Hence, as a consequence of the above relationships between double Ward quasigroups and groups we obtain the following Lemma.

\begin{lemma}\label{L41} In any double Ward quasigroup $(Q,\cdot,e)$ we have

\begin{itemize}
\item[$(1)$] \ $ee=e$,
\item[$(2)$] \ $ex=xe$,
\item[$(3)$] \ $ex\cdot ey=e\cdot yx$,
\item[$(4)$] \ $e(ey\cdot ex)=xy$,
\item[$(5)$] \ $(e\cdot xz)(ey\cdot z)=xy$,
\item[$(6)$] \ $x\cdot xe=ex\cdot x=e$,
\item[$(7)$] \ $xy\cdot x=x\cdot yx=y$
\end{itemize}
for all $x,y,z\in Q$.
\end{lemma}
                                                                                       
In a Ward quasigroup $(Q,\cdot,e)$ we have $ee=e$, but in \eqref{e2} the element $ee$ cannot be replaced by $e$, i.e., a quasigroup $(Q,\cdot,e)$ satisfying the identity $(e\cdot xz)(ey\cdot z)=xy$ may not be a double Ward quasigroup. As an example of such quasigroup we can consider a quasigroup defined by table
$${\small
\begin{array}{c|cccccc}
&1&2&3&4&5&6\\ \hline
1&3&6&1&5&4&2\\
2&6&5&4&3&2&1\\
3&1&4&6&2&5&4\\
4&5&3&2&6&1&4\\
5&4&2&5&1&3&6\\
6&2&1&3&4&6&5
\end{array}}
$$
or the quasigroup $(Q,\cdot)$ with the operation $xy=b-x-y$, where $(Q,+,0)$ is an abelian group and $b\ne 0$. For $b=0$ it is a double Ward quasigroup.

\begin{proposition}\label{P42}
If $(Q,\cdot,e)$ is a Ward quasigroup, then $(Q,\star,e)$ with the operation $x\star y = ex\cdot y$ is a double Ward quasigroup denoted by {\rm D}$(Q,\cdot,e)$.
\end{proposition}
\begin{proof} 
Since $(Q,\cdot,e)$ is a quasigroup, it follows readily that $(Q,\star,e)$ is a quasigroup. Moreover, $e\star x = ee\cdot x=ex$. Thus, $e\star e=ee=e$ and
$$
(e\star (x\star z))\star ((e\star y)\star z)=e(ex\cdot z)\star (e\cdot ey)z\stackrel{\eqref{e9}}{=}
(ex\cdot z)\cdot yz\stackrel{\eqref{e1}}{=}ex\cdot y=x\star y.
$$
So, $((e\star e)\star (x\star z))\star ((e\star y)\star z)=x\star y$. Therefore $(Q,\star,e)$ is a double Ward quasigroup.
\end{proof}
\begin{proposition}\label{P43}
If $(Q,\star,e)$ is a double Ward quasigroup, then $(Q,\bullet,e)$ with the operation $x\bullet y = (e\star x)\star y$ is a Ward quasigroup denoted by {\rm D}$(Q,\star,e)$.
\end{proposition}
\begin{proof}
It is clear that $(Q,\bullet, e)$ is a quasigroup. In this quasigroup 
$(x\bullet z)\bullet (y\bullet z)=((e\star x)\star z)\bullet ((e\star y)\star z)=(e\star x)\star y=x\bullet y$, by Lemma \ref{L41}$(5)$. So, it is a Ward quasigroup.
\end{proof}
\begin{theorem}\label{T44} 
{\rm D(D}$(Q,\cdot,e))=(Q,\cdot,e)$ and {\rm D(D}$(Q,\star,e))=(Q,\star,e)$.
\end{theorem}
\begin{proof} 
Since $e\star y=ex$, we have 
$x\bullet y=(e\star x)\star y=ex\star y=(e\cdot ex)y\stackrel{\eqref{e9}}{=}xy$. So, 
{\rm D(D}$(Q,\cdot,e))=(Q,\cdot,e)$. Analogously, by Lemma \ref{L41}, for $x\bullet y=(e\star x)\star y$ we have 
$x\diamond y=(e\bullet x)\bullet y=(e\star x)\bullet y=(e\star (e\star x))\star y=(e\star (x\star e))\star ((e\star y)\star e)=x\star y$. This proves {\rm D(D}$(Q,\star,e))=(Q,\star,e)$.
 \end{proof}

\begin{theorem}\label{T45} {\rm ret(D}$(Q,\cdot,e))={\rm Ret}(Q,\cdot,e)$  and  {\rm ret(D}$(Q,\star,e))={\rm Ret}(Q,\star,e)$.
\end{theorem}
\begin{proof} In D$(Q,\cdot,e)$ we have $x\star y=ex\cdot y$. So, for ret(D$(Q,\cdot,e))$ we obtain $x\diamond y=x\star (e\star y)=x\star (ee\cdot y)=x\star ey=ex\cdot ey$. Thus ret(D$(Q,\cdot,e))={\rm Ret}(Q,\cdot,e)$.

In D$(Q,\star,e)$ we have $x\bullet y=(e\star x)\star y$. Hence, by Lemma \ref{L41}, in ret(D$(Q,\star,e))$, $x\otimes y=x\bullet (e\bullet y)=x\bullet ((e\star e)\star y)=x\bullet (e\star y)=(e\star x)\star (e\star y)$.
This implies ret(D$(Q,\star,e))={\rm Ret}(Q,\star,e)$.    
\end{proof}

Given Theorem \ref{T34}, we might hope that a double Ward quasigroup $(Q,\cdot,e)$ forms a double magma with the group it induces, Ret$(Q,\cdot,e)$, but this is not the case. However, the lateral inverse law does hold; that is, $xy\diamond zw=(y\diamond x)(w\diamond  z)$. 

\begin{theorem}\label{T46} Any double Ward quasigroup $(Q,\cdot,e)$ forms a lateral double magma with the group {\rm Ret}$(Q,\cdot,e)$. If a double Ward quasigroup $(Q,\cdot,e)$ forms a lateral double magma with a group $(Q,\circ,e)$, then $(Q,\circ,e)={\rm Ret}(Q,\cdot,e)$.
\end{theorem}
\begin{proof} Let $(Q,\cdot,e)$ be a double Ward quasigroup. Then, by Lemma \ref{L41} (3), we have 
$$
xy\diamond zw=(e\cdot xy)(e\cdot zw)=(ey\cdot ex)(ew\cdot ez)= (y\diamond x)(w\diamond  z).
$$
So, $(Q,\cdot,\diamond)$ is a lateral double magma.

If $(Q,\cdot,e)$ forms a lateral double magma with a group $(Q,\circ,e)$, then 
$$
x\circ y=(ex\cdot e)\circ (e\cdot ey)=(e\circ ex)(ey\circ e)=ex\cdot ey=x\diamond y,
$$
by Lemma \ref{L41}$(7)$ and the identity \eqref{e4}. 
\end{proof}

Note that in general from the fact that a double Ward quasigroup $(Q,\cdot,e)$ forms a lateral double magma with a group $(Q,\circ,1)$ it does not follows that $(Q,\circ,1)={\rm Ret}(Q,\circ,1)$. As an example we can consider the additive group $(\mathbb{Z}_6,+,0)$ and the double Ward quasigroup $(\mathbb{Z}_6,\ast,2)$, where $x\ast y=(-x-y)({\rm mod}\,6)$. Then $(\mathbb{Z}_6,\ast,+)$ is a lateral double magma but $(\mathbb{Z}_6,\ast,2)\ne{\rm Ret}(\mathbb{Z}_6,+,0)$.

\begin{corollary}\label{C47} If $(Q,\cdot,e)$ is a Ward quasigroup, then $(Q,\star,e)={\rm D}(Q,\cdot,e)$ and $(Q,\diamond,e)={\rm Ret(D}(Q,\cdot,e)$ form a lateral double magma $(Q,\star,\diamond)$. Moreover, in this case $x\diamond y=x\circ y=x\cdot ey$.
\end{corollary}

As a consequence of the above relationships between double Ward quasigroups, Ward quasigroups and groups we obtain the following corollary.

\begin{corollary}\label{C48}
A double Ward quasigroup is commutative if and only if it is medial if and only if the Ward quasigroup inducing it is medial, or equivalently if its retract is commutative.
\end{corollary}

Given Theorem \ref{T32}, we may wonder if a double Ward quasigroup and its dual form a double magma or a reverse double magma. This question is solved below.

\begin{theorem}\label{T49} The dual of a double Ward quasigroup is a double Ward quasigroup.
\end{theorem}
\begin{proof}  Suppose that $(Q,\cdot,e)$ is a double Ward quasigroup. Then for is its dual quasigroup $(Q,\,\bar{\cdot}\,,e)$, by Lemma \ref{L41}, we have $e\,\bar{\cdot}\,e=e$ and
$$
\arraycolsep=.5mm
\begin{array}{rll}
(e\,\bar{\cdot}\,(x\,\bar{\cdot}\,z))\,\bar{\cdot}\,((e\,\bar{\cdot}\,y)\,\bar{\cdot}\,z)&=
(z\cdot ye)\cdot (zx\cdot e)=e(y\cdot ez)\cdot (e\cdot zx)\\
=e(y\cdot ez)\cdot (ex\cdot ez)
&\stackrel{\eqref{e2}}{=}y\cdot x=x\,\bar{\cdot}\,y.
\end{array}
$$  
Thus $(Q,\,\bar{\cdot}\,,e)$ is a double Ward quasigroup.
\end{proof}

\begin{corollary}\label{C411} A double Ward quasigroup $(Q,\cdot,e)$ forms a double magma with its dual $(Q,\,\bar{\cdot}\,,e)$  if and only if it forms a lateral double magma with its dual if and only if $(Q,\cdot,e)$ is commutative if and only if $(Q,\cdot,e)$ is medial.
\end{corollary} 

\begin{theorem}\label{T4.10} A magma $(Q,\cdot,e)$ is a double Ward quasigroup if and only if it is cancellative and satisfies the identity $\eqref{e2}$.
\end{theorem}
\begin{proof} Since a double Ward quasigroup is cancellative and satisfies \eqref{e2}, we will prove only the converse statement. For this purpose assume that a magma $(Q,\cdot,e)$ is cancellative and satisfies \eqref{e2}.   

 First, we prove that $ee=e$. 

By \eqref{e2}, we have 
\begin{eqnarray}\label{a}
&&ee=(ee\cdot ee)(ee\cdot e).
\end{eqnarray}
Moreover, for $x=ee\cdot ee$ and $y=z=ee\cdot e$, we obtain
$$
(ee\cdot ee)(ee\cdot e)\stackrel{\eqref{e2}}{=}\Big(ee\cdot (ee\cdot ee)(ee\cdot e)\Big)\Big(e(ee\cdot e)\cdot (ee\cdot e)\Big)\stackrel{\eqref{a}}{=}\Big(ee\cdot ee\Big)\Big(e(ee\cdot e)\cdot (ee\cdot e)\Big),
$$
whence, by left cancellation, we get 
$$
ee\cdot e=(e(ee\cdot e))(ee\cdot e).
$$
From \eqref{e2}, for $x=ee$, $y=z=e$, we also obtain $(ee\cdot (ee\cdot e))(ee\cdot e)=ee\cdot e$. Thus, 
$$
(e(ee\cdot e))(ee\cdot e)=(ee\cdot (ee\cdot e))(ee\cdot e)
$$
This, by right cancellation, implies $e=ee$.

So, $(Q,\cdot,e)$ satisfies the identity $(e\cdot xz)(ey\cdot z)=xy$. Putting in this identity $x=ex$ and $y=z=e$ we have
$$
ex\cdot e=(e\cdot (ex\cdot e))(ee\cdot e)=(e\cdot (ex\cdot e))\cdot e.
$$
From this, in view of cancellativity, we deduce $x=ex\cdot e$.

Analogously, putting $x=ee$, $y=xe$ and $z=e$ we obtain $x=e\cdot xe$. Therefore,
\begin{eqnarray}\label{b}
&&x=ex\cdot e=e\cdot xe.
\end{eqnarray}

Now we are able to show that the equations $ax=b$ and $ya=b$ have solutions.
The first equation is solved by $x=b(ae\cdot e)$, the second by $y=e(eb\cdot ae)$. Indeed,
\begin{eqnarray*}
&ax=a\Big(b(ae\cdot e)\Big)\stackrel{\eqref{b}}{=}
\Big(e\cdot e(ae\cdot e)\Big)\Big((e\cdot be)(ae\cdot e)\Big)\stackrel{\eqref{e2}}{=}e\cdot be=b,
\\
&ya=e(eb\cdot ae)\cdot a\stackrel{\eqref{b}}{=}e(eb\cdot ae)\cdot (ee\cdot ae)\stackrel{\eqref{e2}}{=}eb\cdot e=b.
\end{eqnarray*}
Since $(Q,\cdot,e)$ is cancellative, these solutions are unique. Hence $(Q,\cdot,e)$ is a double Ward quasigroup.
\end{proof}

\begin{corollary}\label{C412} A magma $(Q,\cdot)$ is a double Ward quasigroup if and only if it is cancellative and has an idempotent $e$ such that $(e\cdot xz)(ey\cdot z)=xy$ is valid for all $x,y,z\in Q$.
\end{corollary}

Let $Q$ be a finite set. For simplicity we assume that they have form $Q=\{1,2,\ldots,n\}$ with the natural ordering $1,2,\ldots,n$, which is always possible by renumeration of elements. Moreover, instead of $i\equiv j({\rm mod}\,n)$ we will write $[i]_n=[j]_n$. Additionally, in calculations of modulo $n$, we assume that $0=n$. 

Recall that a magma $(Q,\cdot)$ is {\em $k$-translatable}, where $1\leqslant k<n$, if its multiplication table is obtained by the following rule: If the first row of the multiplication table is $a_1,a_2,\ldots,a_n$, then its $s$-th row is obtained from the $(s-1)$-st row by taking the last $k$ entries in the $(s-1)$-st row and inserting them as the first $k$ entries of the $s$-th row and by taking the first $n-k$ entries of the $(s-1)$-st row and inserting them as the last $s-k$ entries of the $s$-th row, where $s\in\{2,3,\ldots,n\}$. Then the (ordered) sequence $a_1,a_2,\ldots,a_n$ is called a {\em $k$-translatable sequence} of $(Q,\cdot)$ with respect to a given ordering. Note that a magma $(Q,\cdot)$ may be $k$-translatable for one ordering but not for another. 

We will need the following characterization of $k$-translatable magma.

\begin{lemma}{\rm (cf. \cite[Lemma 2.5]{DM2})} A magma $(Q,\cdot)$ is $k$-translatable
if and only if for all $i,j\in Q$ we have $i\cdot j=[i+1]_n\cdot [j+k]_n$, or equivalently, $i\cdot j=a_{[k-ki+j]_n}$, where $a_1,a_2,\ldots,a_n$ is the first row of the multiplication table of $(Q,\cdot)$. 
\end{lemma}

\begin{theorem}\label{T414}
If $(Q,\cdot,e)$ is a $k$-translatable double Ward quasigroup of order $n$, then $k=n-1$ and $(Q,\cdot,e)$ is induced by a cyclic group.
\end{theorem}
\begin{proof} The proof is based on Lemma \ref{L41}.

Let a double Ward quasigroup $(Q,\cdot,e)$ be $k$-translatable with translatable sequence $a_1,a_2,\ldots,a_n$ and $i\cdot j=a_{[k-ki+j]_n}$. We can order $Q$ such that $e=1$. Then since $e\cdot x=x\cdot e$  for all $x\in Q$, $a_2=1\cdot 2=2\cdot 1=a_{[1-k]_n}$. Therefore, $k=-1$  and $i\cdot j=a_{[i+j-1]_n}$.

$(Q,\cdot,e)$ also satisfies the identity $(i\cdot j)\cdot i=j$. Thus, 
$$
n=(1\cdot n)\cdot 1=a_n\cdot 1=a_{a_n}=(2\cdot n)\cdot 2=a_1\cdot 2=1\cdot 2=a_2.
$$
So, $n=a_2$ and $2=a_n$. Moreover, for all $j\in Q$ we have
$$
n-j=[(j+1)\cdot (n-j)]_n\cdot (j+1)=a_n\cdot(j+1)=2\cdot(j+1)=a_{[j+2]_n}.
$$
							
Now, using the above facts we can prove that the group Ret$(Q,\cdot,e)$, i.e., the group that induces $(Q,\cdot,e)$, is cyclic and is generated by the element $n$. 

Firstly, observe that 
$$
i\diamond j=(1\cdot i)\cdot (1\cdot j)=a_i\cdot a_j=a_{[a_i+a_j-1]_n}=a_{[(n-i-2)+(n-j-2)-1]_n}\!=[i+j-1]_n.
$$ 

Next, we prove, by induction on $t$, the hypothesis $n^t=[1-t]_n$.
Clearly, for $t=1$ it is true. Assume that it is true for all $k\leqslant t-1$.  Then, $n^t=n^{t-1}\diamond n=[2-t]_n\diamond n=[(2-t)+n-1]_n=[1-t]_n$. So, it is true for for all $t$. This implies that the ordered $n$-tuples $n,n^2,n^3,\ldots,n^{n-1},n^n$ and $n, n-1,n-2,\ldots,2,1$  are equal, so $n$ generates Ret$(Q,\cdot,e)$.  
\end{proof}
\begin{corollary}\label{C415} A double Ward quasigroup is translatable if and only if it is induced by a cyclic group. Such quasigroup is commutative.
\end{corollary}

\begin{theorem}\label{T417} A finite Ward quasigroup can be $k$-translatable only for $k=1$. Then it is induced by a cyclic group.
\end{theorem}
\begin{proof} Let a $k$-translatable Ward quasigroup $(Q,\cdot,e)$ be ordered in this way that $e=1$ and let $a_1,a_2,\ldots,a_n$ be its translatable sequence. Then $i\cdot j=a_{[k-ki+j]_n}$ and
$1=1\cdot 1=2\cdot 2$, So, $a_{[k-k+1]_n}=a_{[-k+2]_n}$. Hence, $1=[2-k]_n$ and $k=1$. 
Therefore, $i\cdot j=a_{[1-i+j]_n}$.

Thus, in the group ret$(Q,\cdot,e)$ that induces $(Q,\cdot,e)$ we have 
$$
i\circ j=i\cdot (1\cdot j)=i\cdot a_j=a_{[1-i+a_j]_n}.
$$
Consequently, $j=1\circ j=j\circ 1$ implies $a_{a_j}=a_{[2-j]_n}$. So, $a_j=[2-j]_n$.

Next,in the same way as in the proof of Theorem \ref{T414}, we prove that the group ret$(Q,\cdot,e)$ is cyclic and generated by $n$.  
\end{proof}
\begin{corollary} A Ward quasigroup is translatable if and only if it is induced by a cyclic group.
\end{corollary}

\section*{\centerline{5. Duality}}\setcounter{section}{5}\setcounter{theorem}{0} 

By Lemma \ref{L23} $(4)$ and Theorem \ref{T31} a pair of unipotent, right modular and left unital quasigroups $(Q,\cdot)$ and $(Q,\ast)$ form a double magma if and only if their operations are equal. When can a pair of right modular, left unital magma that are not necessarily unipotent form a double magma?  

In order to solve this problem, we define an involutive mapping $\alpha$ on a set $Q$ as a mapping whose square is the identity map on $Q$. Clearly, an involutive mapping is a bijection.

\begin{theorem}\label{T51} A pair of right modular, left unital magma $(Q,\cdot,e)$ and $(Q,\ast,\overline{e})$ form a double magma $(Q,\cdot,\ast)$ if and only if there is an involutive automorphism $\alpha$ of $(Q,\ast)$ such that $x\ast y=\alpha x\cdot y$. In this case, $\alpha x=(x\ast \overline{e})\cdot\overline{e}=(x\cdot\overline{e})\ast\overline{e}$, where $\overline{e}$ is a left unit of $(Q,\ast)$.
 \end{theorem}
\begin{proof}  Let $(Q,\cdot)$ and $(Q,\ast)$ be right modular magma with left units $e$ and $\overline{e}$, 
respectively and let $(Q,\cdot,\ast)$ be a double magma. Then, 
$$
\overline{e}=e\cdot\overline{e}=(\overline{e}\ast e)\cdot (\overline{e}\ast\overline{e})\stackrel{\eqref{e3}}{=}(\overline{e}\cdot\overline{e})\ast (e\cdot\overline{e})=(\overline{e}\cdot\overline{e})\ast \overline{e}.
$$
Thus,
$$\overline{e}=\overline{e}\ast\overline{e}=((\overline{e}\cdot\overline{e})\ast\overline{e})\ast\overline{e}=\overline{e}\cdot\overline{e},
$$ because $(Q,\ast)$ is right modular.

Therefore, $\overline{e}\cdot e=(\overline{e}\cdot\overline{e})\cdot e=\overline{e}\cdot\overline{e}=\overline{e}$. So, $\overline{e}=\overline{e}\ast\overline{e}=(e\cdot \overline{e})\ast(\overline{e}\cdot e)\stackrel{\eqref{e3}}{=}(e\ast\overline{e})\cdot (\overline{e}\ast e)=(e\ast\overline{e})\cdot e$. 
Then, $\overline{e}=\overline{e}\cdot e=((e\ast\overline{e})\cdot e)\cdot e=e\ast \overline{e}$. Consequently,
$\overline{e}=\overline{e}\ast\overline{e}=(e\ast\overline{e})\ast\overline{e}=e$.

Thus, by right modularity, $x\cdot y=((x\ast\overline{e})\ast\overline{e})\cdot (\overline{e}\ast y)\stackrel{\eqref{e3}}{=}((x\ast \overline{e})\cdot \overline{e})\ast y=\alpha x\ast y$, 
where 
$\alpha x=(x\ast \overline{e})\cdot\overline{e}=(x\ast\overline{e})\cdot (\overline{e}\ast\overline{e})\stackrel{\eqref{e3}}{=}(x\cdot\overline{e})\ast(\overline{e}\cdot\overline{e})=(x\cdot\overline{e})\ast \overline{e}$.

It is easily to verify that $\alpha$ is an involutive automorphism of $(Q,\ast)$. Hence, $x\ast y=\alpha x\cdot y$.

Conversely, if $x\ast y=\alpha x\cdot y$ and $\alpha$ is an involutive automorphism of $(Q,\ast)$, then, as  it is not difficult to see, $\alpha$ is also an automorphism of $(Q,\cdot)$. Since, right modularity implies mediality, we also have
$(x\ast y)\cdot(z\ast w)=(\alpha x\cdot y)\cdot (\alpha z\cdot w)=(\alpha x\cdot\alpha z)\cdot (y\cdot w)=\alpha(x\cdot z)\cdot (y\cdot w)=(x\cdot z)\ast(y\cdot w).
$ 
So, $(Q,\cdot,\ast)$ is a double magma.
\end{proof}

\begin{corollary}\label{C52}
If $(Q,\cdot,\ast)$ and $(Q,\cdot,\star)$ are right modular, left unital double magmas such that $x\cdot y=\alpha x\ast y=\beta x\star y$, where $\alpha$ and $\beta$ are involutive automorphisms of $(Q,\cdot)$, then $(Q,\ast, \star)$ is a right modular, left unital double magma if and only if $\alpha$ and $\beta$ commute, or equivalently, $(x\ast e)\star e=(x\star e)\ast e$ for all $x\in Q$, where $e$ is the unique left unit of $(Q,\cdot)$.
\end{corollary}
\begin{corollary}\label{53} $(Q,\cdot,\,\bar{\cdot}\,)$ is a double magma if $(Q,\cdot)$ is a right modular, left unital magma.
\end{corollary}

The following theorem answers the question as to whether a right modular, left unital magma $(Q,\cdot,e)$ can form a double magma with a left modular, right unital magma other than $Q,\,\bar{\cdot}\,,e)$.

\begin{theorem}\label{T54} Suppose that $(Q,\cdot,e)$ is left modular and right unital and that $(Q,\ast,\hat{e})$ is right modular and left unital. Then $(Q,\cdot,\ast)$ is a double rl-unital magma if and only if there is an involutive automorphism $\alpha$ of $(Q,\ast)$ such that $x\ast y=y\cdot\alpha x$.    
\end{theorem}
\begin{proof}
$(\Rightarrow)$. First observe that  
$$
e=\hat{e}\ast e=(\hat{e}\cdot e)\ast(e\cdot e)\stackrel{\eqref{e3}}{=}(\hat{e}\ast e)\cdot(e\ast e)=e\cdot(e\ast e).
$$
This, by left and right modularity, gives 
$$
e=e\cdot e=e\cdot(e\cdot(e\ast e))=e\ast e=(\hat{e}\ast e)\ast e=(e\ast e)\ast\hat{e}=e\ast \hat{e}.
$$
Hence, $e=e\ast e=e\ast\hat{e}$. Thus,  
$$\arraycolsep=.5mm
\begin{array}{rll}
e&=\hat{e}\ast e=(e\cdot(e\cdot\hat{e}))\ast(e\cdot e)\stackrel{\eqref{e3}}{=}(e\ast e)\cdot((e\cdot\hat{e})\ast e)=(e\ast\hat{e})\cdot ((e\cdot\hat{e})\ast e)\\
&\stackrel{\eqref{e3}}{=}(e\cdot (e\cdot\hat{e}))\ast(\hat{e}\cdot e)=\hat{e}\ast\hat{e}=\hat{e}.
\end{array}
$$
So, 
$$\arraycolsep=.5mm
\begin{array}{rll}
x\ast y&=(e\cdot(e\cdot x))\ast(y\cdot e)\stackrel{\eqref{e3}}{=}(e\ast y)\cdot ((e\cdot x)\ast e)= (\hat{e}\ast y)\cdot ((e\cdot x)\ast e)\\[4pt]
&=y\cdot ((e\cdot x)\ast e)=y\cdot\alpha x.
\end{array}
$$
It is easily prove that $\alpha x=(e\cdot x)\ast e$ is an involutive automorphism of $(Q,\ast)$.

$(\Leftarrow)$.
 If $\alpha$ is an involutive automorphism of $(Q,\ast)$ such that $x\ast y=y\cdot\alpha x$, then $\alpha$ is also an involutive automorphism of $(Q,\cdot)$. Then, 
$$
(x\ast y)\cdot(z\ast w)=(y\cdot\alpha x)\cdot (w\cdot\alpha z)=(y\cdot w)\cdot(\alpha x\cdot\alpha z)=(y\cdot w)\cdot\alpha(x\cdot z)=(x\cdot z)\ast(y\cdot w)
$$  
and so $(Q,\cdot,\ast)$ is a double $rl$-unital magma. 
\end{proof} 

\begin{corollary}\label{C55}
Suppose that $(Q,\cdot,e)$ is right modular and left unital, $(Q,\ast,e)$ and $(Q,\star,e)$ are left modular, right unital magmas such that $(Q,\cdot,\ast)$ and $(Q,\cdot,\star)$ are double magmas, with $x\ast y=\alpha y\cdot x$, $x\star y=\beta y\cdot x$, where $\alpha,\beta$ are involutive automorphism of $(Q,\cdot)$. Then $(Q,\ast,\star)$ is a double magma if and only if $\alpha$ and $\beta$ commute, or equivalently, $(x\ast e)\star e=(x\star e)\ast e$ for all $x\in Q$, where $e$ is the unique left unit of $(Q,\cdot)$.
\end{corollary}

\begin{example} Let $(Q,\cdot)$ and $(Q,\ast)$ be idempotent magmas. Then $(Q,\cdot,\ast)$ is a lateral double magma if and only if $(Q,\cdot)=(Q,\ast)$, and $(Q,\cdot,\ast)$ is a reversible double magma if and only if $(Q,\ast)=(Q,\,\bar{\cdot}\,)$.
\end{example}

\section*{\centerline{6. Double magma partners of groups}}\setcounter{section}{6}\setcounter{theorem}{0} 

Let $(Q,\circ,e)$ be a group. Theorem \ref{T26} implies that for any unital double magma partner $(Q,\ast,\hat{e})$ of $(Q,\circ,e)$, $(Q,\circ)=(Q,\ast)$ is abelian. So, if $(Q,\circ,e)$ is not abelian, then it has no unital double magma partners. If $(Q,\circ,e)$ is abelian, then its only double magma partner is itself. What about right unital or left unital double magma partners of $(Q,\circ,e)$? The proof of the following Lemma is straightforward and is omitted.
\begin{lemma}\label{L61} If a group $(Q,\circ,e)$ and a magma $(Q,\ast,\hat{e})$ are double magma partners, then
 \begin{enumerate}
\item[$(i)$] if $\hat{e}$ is a unique right or left unit of $(Q,\ast,\hat{e})$, then $e=\hat{e}$,                          
\item[$(ii)$] if $\hat{e}$ is a right unit of $(Q,\ast,\hat{e})$ and  $(Q,\ast,\hat{e})$ is left cancellative, then $e=\hat{e}$,                               
\item[$(iii)$] if $\hat{e}$ is a left unit of $(Q,\ast,\hat{e})$ and $(Q,\ast,\hat{e})$ is right cancellative, then $e=\hat{e}$,    
\item[$(iv)$] if $\hat{e}$ is a right or left unit of $(Q,\ast,\hat{e})$ and $(Q,\ast,\hat{e})$ is unipotent, then $e=\hat{e}$.
\end{enumerate}
	\end{lemma}

Recall that if $(Q,\cdot,e)$ is a unipotent, left unital Ward-dual quasigroup, then $\overline{\rm re}{\rm t}(Q,\cdot,e)=(Q,\bar{\circ},e)$, where $x\,\bar{\circ}\, y=xe\cdot y=y\,\bar{\cdot}\,(e\,\bar{\cdot}\, x)$ is a group with the unity $e$. So, $(Q,\,\bar{\cdot}\,,e)={\rm der}(Q,\bar{\circ},e)$. Hence, $xy=x^{-1}\circ y$ and $x^{-1}=xe$. On the other side, if $(Q,\circ,e)$ is an abelian group, then ${\rm d}\overline{\rm {er}}(Q,\circ,e)=(Q,\cdot)$, where $xy=x^{-1}\circ y$ and $x^{-1}=x\circ e$, is a right modular, unipotent quasigroup with unique left unit $e$. Moreover, $(Q,\circ,e)$ and ${\rm d}\overline{\rm {er}}(Q,\circ,e)$ are double magma partners. We now prove that ${\rm d}\overline{\rm {er}}(Q,\circ,e)$  is the only right modular unipotent quasigroup with left unit that is a double magma partner of a group.

\begin{theorem}\label{T62} Let $(Q,\circ,e)$ be any group. If $(Q,\circ,e)$ and $(Q,\star,\hat{e})$ are double magma partners and $(Q,\star,\hat{e})$ is a right modular unipotent magma with left unit $\hat{e}$, then $(Q,\circ,e)$ is abelian and $(Q,\star,\hat{e})={\rm d}\overline{\rm {er}}(Q,\circ,e)$.
\end{theorem}
\begin{proof} Since $(Q,\star,\hat{e})$ is a right modular magma, its left unit $\hat{e}$ is unique. So, by Lemma \ref{L61}, $e=\hat{e}$. Then, $x\circ y=((x\star e)\star e)\circ (e\star y)=((x\star e)\circ e)\star (e\circ y)=(x\star e)\star y=(y\star e)\star x=y\circ x$. Unipotency gives $x\circ(x\star e)=e$ and so $x^{-1}=x\star e$. Hence, $x\star y=((x\star e)\star e)\star y=(x\star e)\circ y=x^{-1}\circ y=xy$. 
\end{proof}

\begin{lemma}\label{L63} If a unital magma $(Q,\cdot,e)$ forms a double magma with a quasigroup $(Q,\ast)$, then this magma is commutative.
\end{lemma}
\begin{proof}  For any $w,z\in Q$ there exist $x,y\in Q$ such that $x\ast e=w$ and $e\ast y=z$. Then, 
$zw=(e\ast y)(x\ast e)\stackrel{\eqref{e3}}{=}ex\ast ye=xe\ast ey\stackrel{\eqref{e3}}{=}(x\ast e)(e\ast y)=wz$.
\end{proof}

\begin{lemma}\label{L64} A magma $(Q,\cdot)$ is a Ward quasigroup if and only if it is cancellative and satisfies \eqref{e1}.
\end{lemma}
\begin{proof} Suppose that a cancellative magma $(Q,\cdot)$ satisfies \eqref{e1}, i.e $xz\cdot yz=xy$ for all $x,y,z\in Q$. Then, in particular, $xx\cdot xx=xx$. Therefore, for any $y\in Q$, will be
$y\cdot xx\stackrel{\eqref{e1}}{=}(y\cdot xx)(xx\cdot xx)=(y\cdot xx)\cdot xx$. Right cancellation implies that $y=y\cdot xx$ for any $x,y\in Q$. This implies that $xx=yy$. 

Let $xx=e$. Then, $y=ye$, 
\begin{eqnarray}
&xy\cdot z\stackrel{\eqref{e1}}{=}(xy\cdot ey)(z\cdot ey)\stackrel{\eqref{e1}}{=}xe\cdot (z\cdot ey)=x(z\cdot ey)\rule{20mm}{0mm}\label{e13}\\
{\rm and }\rule{20mm}{0mm}&\nonumber\\
&\label{e14}
xy\stackrel{\eqref{e1}}{=}xx\cdot yx =e\cdot yx.
\end{eqnarray}

Using the two identities we can show that equations $ax=b$ and $ya=b$ have solutions for each $a,b\in Q$. The first equation has the solution $x=eb\cdot ea$, the second $y=b\cdot ea$. Indeed, $ax=a(eb\cdot ea)\stackrel{\eqref{e13}}{=}aa\cdot eb=e\cdot eb\stackrel{\eqref{e14}}{=}be=b$. Analogously,
$ya=(b\cdot ea)a\stackrel{\eqref{e13}}{=}b(a(e\cdot ea))\stackrel{\eqref{e14}}{=}b\cdot aa=be=b$. Since $(Q,\cdot)$ is cancellative solutions are unique. So, $(Q,\cdot)$ is a Ward quasigroup.

The converse statement is obvious.
\end{proof}

\begin{theorem}\label{T65} If the group $(Q,\circ,e)$ forms a double magma with a cancellative magma $(Q,\cdot)$ that satisfies the identity $zx\cdot zy=xy$, then $(Q,\circ)$ is a boolean group and $(Q,\circ)=(Q,\cdot)$.
\end{theorem}
\begin{proof} By Lemma \ref{L64}, $(Q,\,\bar{\cdot}\,)$ is a Ward quasigroup. By Lemma \ref{L63}, $(Q,\cdot)$ is commutative and so it is unital. By Theorem \ref{T26} $(Q,\circ)$=$(Q,\cdot)$ is a group. By \eqref{e6} $xx=e$ for all $x\in Q$ and so $(Q,\circ,e)$ is a boolean group.
\end{proof}

We can now state without proof the theorems dual to Theorems \ref{T34} and \ref{T35} as follows:

\begin{theorem}\label{T66} If $(Q,\cdot,e)$  is a unipotent, left unital and right modular quasigroup, then   $(Q,\cdot,\,\bar{\circ}\,)$, where $(Q,\bar{\circ},e)=\overline{\rm re}{\rm t}(Q,\cdot,e)$ is a double magma. Moreover, $(Q,\cdot,e)$ satisfies the lateral inverse law if and only if it satisfies the reverse inverse law if and only if $\,\overline{\rm re}{\rm t}(Q,\cdot,e)$ is a boolean group.
\end{theorem}

\begin{theorem}\label{T67} If a Ward-dual quasigroup $(Q,\cdot,e)$ forms a double magma with a unital magma $(Q,\diamond,\hat{e})$, then $(Q,\diamond,\hat{e})=\,\overline{\rm re}{\rm t}(Q,\cdot,e)$ and $(Q,\cdot,e)$ is medial. If $(Q,\cdot,e)$ is medial and its retract $\,\overline{\rm re}{\rm t}(Q,\cdot,e)$ forms a double magma with a right cancellative, unipotent, left unitary magma $(Q,\star)$, then $(Q,\cdot)=(Q,\star)$.
\end{theorem}

\section*{\centerline{7. Double magma and medial quasigroups}}\setcounter{section}{7}\setcounter{theorem}{0}

The theorem proved in this section gives further evidence of the intimate connection between the property of mediality and the interchange law. The theorem states that a quasigroup is medial if and only if any pair of its parastrophic binary operations satisfy the interchange law. Hence, if the interchange law holds between any two of the parastrophic operations of a quasigroup, that fact has a powerful influence on the structure of the quasigroup, via the Toyoda theorem \cite{Toy}. This well-known theorem states that a medial quasigroup $(Q,\cdot)$ can be presented in the form $xy=\alpha x+\beta y+c$, where $(Q,+)$ is an abelian group, $\alpha,\beta$ commuting automorphisms of $(Q,+)$ and $c\in Q$ is a some fixed element.

Each quasigroup $(Q,\cdot)$ determines five new quasigroups $(Q,\circ_i)$ with the operations $\circ_i$ defined as follows:
$$
\begin{array}{cccc}
x\circ_1 y=z\;\Leftrightarrow\; xz=y\\
x\circ_2 y=z\;\Leftrightarrow\; zy=x\\
x\circ_3 y=z\;\Leftrightarrow\; zx=y\\
x\circ_4 y=z\;\Leftrightarrow\; yz=x\\
x\circ_5 y=z\;\Leftrightarrow\; yx=z\\
\end{array}
$$
Such defined (not necessarily distinct) quasigroups are called {\em parastrophes} or {\em conjugates} of a quasigroup $(Q,\cdot)$. Note that parastrophes are pairwise dual, namely $\bar{\cdot}\,=\circ_5$, $\bar{\circ}_1=\circ_4\,$ and $\,\bar{\circ}_2=\circ_3\,$. 

Generally, parastrophes does not save properties of initial quasigroup. Parastrophes of a group are not a group, but parastrophes of an idempotent quasigroup also are idempotent quasigroups. Moreover, in some cases (described in \cite{Lin}) parastrophes of a given quasigroup are pairwise equal or all are pairwise distinct. In \cite{D'15} it is proved that the number of non-isotopic parastrophes of a quasigroup is always a divisor of $6$ and does not depend on the number of elements of a quasigroup.

\begin{theorem}\label{T71} A quasigroup $(Q,\cdot)$ is medial if and only if any pair of its parastrophic binary operations satisfy the interchange law.
\end{theorem}
\begin{proof} Directly from the Toyoda theorem it follows that if a quasigroup $(Q,\cdot)$ is medial, then all its parastrophes are medial, and conversely, if one of the parastrophes of $(Q,\cdot)$ is medial then $(Q,\cdot)$ and its other parastrophes are medial. So, by Lemma \ref{L23}$(7)$, a quasigroup $(Q,\cdot)$ is medial if and only if $(Q,\cdot,\cdot)$ or $(Q,\circ_i,\circ_i)$, $i=1,2,3,4,5$, satisfy the interchange law. By Lemma \ref{L23}$(1)$, a quasigroup $(Q,\cdot)$ is medial if and only if $(Q,\cdot,\circ_5)$ satisfies the interchange law. Also, it is not difficult to see that $(Q,\cdot,\circ_5)$ satisfies the interchange law if and only if $(Q,\circ_k,\circ_5)$, where $k=1,2,3,4$, satisfies this law. Moreover, $(Q,\ast,\star)$ satisfies the interchange law if and only if $(Q,\star,\ast)$ satisfies this law, so our proof can be restricted to the cases $(Q,\cdot,\circ_k)$ and $(Q,\circ_i,\circ_j)$, where $k=1,2,3,4$ and $1\leqslant i<j\leqslant 4$.  

\smallskip
\noindent
$(\Rightarrow)$. Let a quasigroup $(Q,\cdot)$ be medial.\\[2pt]
$\bullet$ $(Q,\cdot,\circ_1)$. Suppose $xy=A$, $zw=B$, $A\circ_1 B=C$, $x\circ_1 z=D$, $y\circ_1 w=E$ and $DE=F$. Then, $AC=B$, $xD=z$, $yE=w$. So, $AC=B=zw=xD\cdot yE=xy\cdot DE=AF$. Thus, $C=F$, which means that $(Q,\cdot,\circ_1)$ is a double magma.\\[2pt]
$\bullet$ $(Q,\cdot,\circ_2)$. Suppose $xy=A$, $zw=B$, $A\circ_2 B=C$, $x\circ_2 z=D$, $y\circ_2 w=E$ and $DE=F$. Then, $CB=A$, $Dz=x$, $Ew=y$. So, $CB=A=xy=Dz\cdot Ew=DE\cdot zw=FB$. Thus, $C=F$, and $(Q,\cdot,\circ_2)$ is a double magma.\\[2pt]
$\bullet$ $(Q,\cdot,\circ_3)$. Suppose $xy=A$, $zw=B$, $A\circ_3 B=C$, $x\circ_3 z=D$, $y\circ_3 w=E$ and $DE=F$. Then, $CA=B$, $Dx=z$, $Ey=w$. So, $CA=B=zw=Dx\cdot Ey=DE\cdot xy=FA$. Thus, $C=F$ and $(Q,\cdot,\circ_3)$ is a double magma.\\[2pt]
$\bullet$ $(Q,\cdot,\circ_4)$. Suppose $xy=A$, $zw=B$, $A\circ_4 B=C$, $x\circ_4 z=D$, $y\circ_4 w=E$ and $DE=F$. Then, $BC=A$, $zD=x$, $wE=y$. So, $BC=A=xy=zD\cdot wE=zw\cdot DE=BF$. Thus, $C=F$ and $(Q,\cdot,\circ_4)$ is a double magma.\\[2pt]
$\bullet$ $(Q,\circ_1,\circ_2)$. Let $x\circ_1 y=A$, $z\circ_1 w=B$, $A\circ_2 B=C$, $x\circ_2 z=D$, $y\circ_2 w=E$ and $D\circ_1 E=F$. Then, $xA=y$, $zB=w$, $CB=A$, $Dz=x$, $Ew=y$ and $DF=E$. So, $y=xA=Dz\cdot CB=DC\cdot zB=DC\cdot w=Ew$ and $E=DC=DF$. Hence, $C=F$. Consequently, $(Q,\circ_1,\circ_2)$ is a double magma.\\[2pt]
$\bullet$ $(Q,\circ_1,\circ_3)$. Let $x\circ_1 y=A$, $z\circ_1 w=B$, $A\circ_3 B=C$, $x\circ_3 z=D$, $y\circ_3 w=E$ and $D\circ_1 E=F$. Then, $xA=y$, $zB=w$, $CA=B$, $Dx=z$, $Ey=w$ and $DF=E$. So, $w=zB=Dx\cdot CA=DC\cdot xA=DC\cdot y=Ey=DF\cdot y$. Hence, $C=F$ and $(Q,\circ_1,\circ_3)$ is a double magma.\\[2pt]
$\bullet$ $(Q,\circ_1,\circ_4)$. Let $x\circ_1 y=A$, $z\circ_1 w=B$, $A\circ_4 B=C$, $x\circ_4 z=D$, $y\circ_4 w=E$ and $D\circ_1 E=F$. Then, $xA=y$, $zB=w$, $BC=A$, $zD=x$, $wE=y$ and $DF=E$. So, $y=xA=zD\cdot BC=zB\cdot DC=w\cdot DC=wE$ and $DC=E=DF$. Hence, $C=F$ and $(Q,\circ_1,\circ_4)$ is a double magma.\\[2pt]
$\bullet$ $(Q,\circ_2,\circ_3)$.  Let $x\circ_2 y=A$, $z\circ_2 w=B$, $A\circ_3 B=C$, $x\circ_3 z=D$, $y\circ_3 w=E$ and $D\circ_2 E=F$. Then, $Ay=x$, $Bw=z$, $CA=B$, $Dx=z$, $Ey=w$ and $FE=D$. So, $FE\cdot x=Dx=z=Bw=CA\cdot Ey=CE\cdot Ay=CE\cdot x$. Hence, $C=F$ and $(Q,\circ_2,\circ_3)$ is a double magma.\\[2pt]
$\bullet$ $(Q,\circ_2,\circ_4)$.  Let $x\circ_2 y=A$, $z\circ_2 w=B$, $A\circ_4 B=C$, $x\circ_4 z=D$, $y\circ_4 w=E$ and $D\circ_2 E=F$. Then, $Ay=x$, $Bw=z$, $BC=A$, $zD=x$, $wE=y$ and $FE=D$. So, $z\cdot CE=Bw\cdot CE=BC\cdot wE=Ay=x=zD=z\cdot FE$. Hence, $C=F$ and $(Q,\circ_2,\circ_4)$ is a double magma.\\[2pt]
$\bullet$ $(Q,\circ_3,\circ_4)$.  Let $x\circ_3 y=A$, $z\circ_3 w=B$, $A\circ_4 B=C$, $x\circ_4 z=D$, $y\circ_4 w=E$ and $D\circ_3 E=F$. Then, $Ax=y$, $Bz=w$, $BC=A$, $zD=x$, $wE=y$ and $FD=E$. So, $w\cdot FD=wE=y=Ax=BC\cdot zD=Bz\cdot CD=w\cdot CD$. Hence, $C=F$ and $(Q,\circ_3,\circ_4)$ is a double magma.

\medskip
This completes the first part of the proof.

\medskip\noindent
$(\Leftarrow )$. Now we will prove that if a pair of parastrophic operations of a quasigroup $(Q,\cdot)$ forms a double magma, then this quasigroup is medial.

\medskip

Let $xy=A$, $zw=B$, $AB=C$, $xz=D$, $yw=E\,$ and $\,DE=F$.\\[4pt]
$\bullet$ $(Q,\cdot,\circ_1)$. If it is a double magma, then $x\circ_1 A=y$ and $z\circ_1 B=w$. 
Hence $yw=(x\circ_1\! A)(z\circ_1\! B)=xz\circ _1\!AB$. So, $xz\cdot yw=AB=xy\cdot zw$, i.e., $(Q,\cdot)$ is medial.\\[2pt]
$\bullet$ $(Q,\cdot,\circ_2)$. In this case $A\circ_2 y=x$ and $B\circ_2 w=z$. 
Hence $xz=(A\circ_2 y)(B\circ_2 w)=AB\circ_2 yw$ and so $xz\cdot yw=AB=xy\cdot zw$. Thus $(Q,\cdot)$ is medial.\\[2pt]
$\bullet$ $(Q,\cdot,\circ_3)$. In this case $y\circ_3 A=x$ and $w\circ_3 B=z$. 
Hence $xz=(y\circ_3 A)(w\circ_3 B)=yw\circ_3 AB$. So, $xz\cdot yw=AB=xy\cdot zw$. Thus $(Q,\cdot)$ is medial.\\[2pt]
$\bullet$ $(Q,\cdot,\circ_4)$. Then $A\circ_4 x=y$ and $B\circ_4 z=w$. 
Consequently, $yw=(A\circ_4 x)(B\circ_4 z)=AB\circ_4 xz$. So, $xz\cdot yw=AB=xy\cdot zw$. Thus $(Q,\cdot)$ is medial.\\[2pt]
$\bullet$ $(Q,\circ_1,\circ_2)$. Then $y=E\circ_2 w=(D\circ_1 F)\circ_2 (z\circ_1 B)=(D\circ_2 z)\circ_1(F\circ_2 B)=x\circ_1(F\circ_2 B)$. Hence, $xy=F\circ_2 zw$. Thus, $xy\cdot zw=F=DE$, i.e., $(Q,\cdot)$ is medial.\\[2pt]
$\bullet$ $(Q,\circ_1,\circ_3)$. Then $z=x\circ_1 D=(y\circ_3 A)\circ_1 (E\circ_3 F)=(y\circ_1 E)\circ_3(A\circ_1 F)=w\circ_3(A\circ_1 F)$. So, $zw=xy\circ_1 F$. Thus, $xy\cdot zw=F=DE$. So, $(Q,\cdot)$ is medial.\\[2pt]
$\bullet$ $(Q,\circ_1,\circ_4)$. Then $w=y\circ_1 E=(A\circ_4 x)\circ_1 (F\circ_4 D)=(A\circ_1 F)\circ_4(x\circ_1 D)=(A\circ_1 F)\circ_4 z$. So, $zw=A\circ_1 F$. Hence, $xy\cdot zw=F=DE$. Thus $(Q,\cdot)$ is medial.\\[2pt]
$\bullet$ $(Q,\circ_2,\circ_3)$. Then $x=y\circ_3 A=(E\circ_2 w)\circ_3 (C\circ_2 B)=(E\circ_3 C)\circ_2(w\circ_3 B)=(E\circ_3 C)\circ_2 z$. So, $xz=yw\circ_3 AB$, which implies the mediality of $(Q,\cdot)$.\\[2pt]
$\bullet$ $(Q,\circ_2,\circ_4)$. Then $y=E\circ_2 w=(F\circ_4 D)\circ_2 (B\circ_4 z)=(F\circ_2 B)\circ_4(D\circ_2 z)=(F\circ_2 B)\circ_4 x$. So, $A=F\circ_2 B$. Thus $F=AB$, which gives the mediality of $(Q,\cdot)$.\\[2pt]
$\bullet$ $(Q,\circ_3,\circ_4)$. Then $y=w\circ_3 E=(B\circ_4 z)\circ_3 (F\circ_4 D)=(B\circ_3 F)\circ_4(z\circ_3 D)=(B\circ_3 F)\circ_4 x$. So, $xy=B\circ_3 F$, i.e., $F=xy\cdot B$. This means the mediality of $(Q,\cdot)$.

\medskip

 This completes the second part of the proof and, therefore of Theorem \ref{T71}.
\end{proof}

If we look at the parastrophes of a group $(G,\circ,1)$, a Ward quasigroup $(W,\circ,r)$, a Ward-dual quasigroup $(W,\bar{\circ},r)$, a double Ward quasigroup $(DW,\circ,e)$ and a unipotent, left unital, right modular quasigroup $(R,\circ,l)$, then we see the usual suspects that have appeared throughout the sections above. This is shown by the following table.
$${\small
\begin{array}{|c|c|c|c|c|c|}
\hline
\rule{0pt}{12pt}&\;\;(G,\circ,1)\;\;&(W,\circ,r)&(W,\bar{\circ},r)&(DW,\circ,e)&(R,\circ,l)\\[2pt] \hline
\rule{0pt}{12pt}\;\;x\circ_1y\;\;&\;x^{-1}\!\circ\! y\;&\;(r\!\circ\! y)\!\circ\!(r\!\circ\! x)\;&(x\bar{\circ} r)\bar{\circ}y&y\!\circ\! x&\;\;(x\!\circ\! l)\!\circ\! y\;\;\\[2pt] \hline
\rule{0pt}{12pt}x\!\circ_2 y&x\!\circ\! y^{-1}&x\!\circ\! (r\!\circ\! y)&\;(y\bar{\circ}r)\bar{\circ}(x\bar{\circ}r)\;&y\!\circ\! x&x\!\circ\! y\\[2pt] \hline
\rule{0pt}{12pt}x\!\circ_3 y&y\!\circ\! x^{-1}&y\!\circ\!(r\!\circ\! x)&(x\bar{\circ}r)\bar{\circ}(y\bar{\circ}r)&x\!\circ\! y&y\!\circ\! x\\[2pt] \hline
\rule{0pt}{12pt}x\!\circ_4 y&y^{-1}\!\circ\! x&(r\!\circ\! x)\!\circ\! (r\!\circ\! y)&(y\bar{\circ}r)\bar{\circ}x&x\!\circ\! y&(y\!\circ\! l)\!\circ\! x\\[2pt] \hline
\rule{0pt}{11pt}x\circ_5 y&y\!\circ\! x&y\!\circ\! x&y\bar{\circ}x&y\!\circ\! x&y\!\circ\! x\\[2pt] \hline
\end{array}}
$$

Note that a non-commutative group has six different parastrophes. If a group is boolean, then all its parastrophes are equal. In other cases $\circ=\circ_5$, $\circ_1=\circ_3$ and $\circ_2=\circ_4$, and the group has exactly $3$ parastrophes.

In the case of a Ward quasigroup $(W,\circ,r)$ we have three possibilities: 

\smallskip
$(1)$ all parastrophes are equal if and only if $(W,\circ,r)$ is a boolean group, 

\smallskip
$(2)$ there are three parastrophes, $(W,\circ_1)=(W,\circ)$, $(W,\circ_2)=(W,\circ_3)$ and \hspace*{11mm}$(W,\circ_4)=(W,\circ_5)$, if and only if $(W,\circ,r)$ is medial but not a boolean group,

\smallskip
$(3)$ there are six parastrophes if and only if $(W,\circ)$ is not boolean and not medial. 

\medskip
In the case of a dual Ward quasigroup $(W,\bar{\circ},r)$ we have the same three possibilities.

\medskip
Parastrophes of a double Ward quasigroup are equal (then it is commutative) or $\circ=\circ_3=\circ_4\ne\circ_1=\circ_2=\circ_5$, if it is non-commutative.

\medskip
Parastrophes of a unipotent, left unital, right modular quasigroups $(R,\circ,l)$ are equal (then it is a boolean group) or $\circ=\circ_2\ne\circ_1=\circ_4\ne\circ_3=\circ_5$. There are no other possibilities (cf. \cite{D'15} or \cite{Lin}).

\medskip
Note that if $(G,\circ)$ is a group, then the parastrophe $(G,\circ_2)=(G,\bar{\circ}_3)$ is a Ward quasigroup; $(W,\circ_2)={\rm ret}(W,\circ)$ and $(W,\circ_3)={\rm ret}(W,\bar{\circ})$ are groups. Parastrophes of a double Ward quasigroup are double Ward quasigroups, because by Theorem \ref{T46} the dual of a double Ward quasigroup is a double Ward quasigroup. The parastrophes $(R,\circ_1,l)$ and $(R,\circ_4,l)$ are groups. For example, $(R,\circ,l)$, as right modular, is medial which implies the associativity of the operation $\circ_1$. Indeed, $(x\circ_1 y)\circ_1 z=(((x\circ l)\circ y)\circ l)\circ z=((l\circ y)\circ (x\circ l))\circ z=((l\circ x)\circ (y\circ l))\circ z=(x\circ (y\circ l))\circ z=(x\circ (y\circ l))\circ (l\circ z)=(x\circ l)\circ ((y\circ l)\circ z)=x\circ_1 (y\circ_1 z)$.

\medskip
As it is not difficult to see, the parastrophe $(R,\circ_3,l)=(R,\circ_5,l)=(R,\bar{\circ},l))$ is a medial Ward quasigroup.

\medskip
Another consequences of the above table are listed below.

\begin{theorem} 
$(Q,\circ_1,r)$ is a Ward quasigroup if and only if $(Q,\circ,r)$, where $x\circ y=(r\circ_1 y)\circ_1 (r\circ_1 x)$,  is a Ward quasigroup.
\end{theorem}
\begin{theorem} 
$(Q,\circ_2,r)$ is a Ward quasigroup if and only if $x\circ y=x\circ_2 (r\circ_2 y)$ and $(Q,\circ,r)={\rm ret}(Q,\circ_2,r)$.
\end{theorem}
\begin{theorem} 
$(Q,\circ_3,r)$ is a Ward quasigroup if and only if $(Q,\circ,r)$, where $x\circ y=(r\circ_3 x)\circ_3 (r\circ_3 y)$, is a Ward-dual quasigroup.
\end{theorem}
\begin{theorem} $(Q,\circ_4)$ is a Ward quasigroup if and only if $(Q,\circ,r)$, where $x\circ y=y\circ_4 (r\circ_4 x)$ is a group dual to the group {\rm der}$(Q,\circ_4,r)$.
\end{theorem}
\begin{theorem}  
$(Q,\circ_1)$ is a Ward dual quasigroup if and only if $x\circ y=(x\circ_1 r)\circ_1 y$ and $(Q,\circ,r)={\rm Ret}(Q,\circ_1,r)$.
\end{theorem}
\begin{theorem}  
$(Q,\circ_2)$ is a Ward dual quasigroup if and only if $(Q,\circ,r)$, where $x\circ y=(y\circ_2 r)\circ_2 (x\circ_2 r)$,  is a  Ward-dual quasigroup.
\end{theorem}
\begin{theorem} 
$(Q,\circ_3)$ is a Ward dual quasigroup if and only if $(Q,\circ,r)$, where $x\circ y=(y\circ_3 r)\circ_3 x$, is a group dual to the group {\rm der}$(Q,\circ_3,r)$.
\end{theorem}
\begin{theorem} 
$(Q,\circ_4)$ is a Ward dual quasigroup if and only if $(Q,\circ,r)$, where $x\circ y=(x\circ_4 r)\circ_4 (y\circ_4 r)$ is a Ward quasigroup.
\end{theorem}
\begin{theorem} 
$(Q,\circ_1,l)$ is a unipotent, right modular, left unital quasigroup if and only if $(Q,\circ,l)$, where $x\circ y=(y\circ_1 l)\circ_1 x$, is an abelian group.
\end{theorem}
\begin{theorem} 
$(Q,\circ_i)$ is a double Ward quasigroup if and only if $(Q,\circ)=Q,\circ_i)$  
or $(Q,\circ)=(Q,\bar{\circ}_i)$ if and only if $(Q,\circ)$ is a double Ward quasigroup.
\end{theorem}

\footnotesize{\rightline{Received \ February 03, 2019}
\noindent
W.A. Dudek \\
 Faculty of Pure and Applied Mathematics,
 Wroclaw University of Science and Technology\\
 50-370 Wroclaw,  Poland \\
 Email: wieslaw.dudek@pwr.edu.pl\\[4pt]
R.A.R. Monzo\\
Flat 10, Albert Mansions, Crouch Hill, London N8 9RE, United Kingdom\\
E-mail: bobmonzo@talktalk.net

}
\end{document}